\documentclass[11pt]{amsart}
\usepackage[sorted-cites]{amsrefs}
\usepackage[matrix,arrow,tips,curve]{xy}
\usepackage[english]{babel}
\usepackage{amsmath}
\usepackage{amssymb}
\numberwithin{equation}{section}
\usepackage{ mathrsfs }
\theoremstyle{plain}
\usepackage{enumerate}
\usepackage{graphicx}
\usepackage{hyperref}
\usepackage{color}
\DeclareMathAlphabet{\pazocal}{OMS}{zplm}{m}{n}

\usepackage{xcolor}

\newtheorem{theorem}{Theorem} [section]
\newtheorem{lemma}{Lemma}[section]
\newtheorem{proposition}{Proposition}[section]

\newtheorem{cor}{Corollary}[section]

\theoremstyle{remark}
\newtheorem{remark}{Remark}[section]

\newcommand{\dist}{\textrm{dist}}

\hypersetup{pdfpagemode=UseNone} 
\hypersetup{pdfstartview=FitH}

\begin{document}

\title[ CSC Ricci shrinker]{ Rigidity of four-dimensional Gradient shrinking Ricci solitons}



\thanks{The authors are  partially supported by CNPq and Faperj of Brazil.}

\address{Departamento de Matem\'atica Aplicada, Instituto de Matem\'atica e Estat\'\i stica, Universidade Federal Fluminense, S\~ao Domingos,
Niter\'oi, RJ 24210-201, Brazil}
\author[Xu Cheng]{Xu Cheng}

\email{xucheng@id.uff.br}
\author[Detang Zhou]{Detang Zhou}
\address{Departamento de Geometria, Instituto de Matem\'atica e Estat\'\i stica, Universidade Federal Fluminense, S\~ao Domingos,
Niter\'oi, RJ 24210-201, Brazil}
\email{zhoud@id.uff.br}

\newcommand{\M}{\mathcal M}

\begin{abstract} Let $(M, g, f)$ be a $4$-dimensional complete noncompact gradient shrinking  Ricci soliton with the equation  $Ric+\nabla^2f=\lambda g$, where $\lambda$ is a positive real number. We prove that if $M$ has constant scalar curvature $S=2\lambda$, it  must be  a quotient of $\mathbb{S}^2\times \mathbb{R}^2$. Together with the known results, this implies that  a $4$-dimensional complete gradient shrinking  Ricci soliton has constant scalar curvature if and only if it is rigid, that is,  it is  either Einstein,  or a finite quotient of Gaussian shrinking soliton  $\Bbb{R}^4$, $\Bbb{S}^{2}\times\Bbb{R}^{2}$ or $\Bbb{S}^{3}\times\Bbb{R}$.
\end{abstract}

\maketitle
\section{Introduction}

A complete $n$-dimensional  Riemannian manifold $(M, g)$ is said to be a {\it gradient  shrinking Ricci soliton} if there exists
a smooth function $f$ on $M$ such that the equation
\begin{align}\label{soliton}
\text{Ric}+\nabla^2f=\lambda g
\end{align}
holds for some positive  constant $\lambda\in \mathbb{R}$, where $\text{Ric}$ is the Ricci tensor and  $\nabla^2f$ is the Hessian of the function $f$.\\

The function $f$ is called a
{\it potential function} of the gradient shrinking Ricci soliton. Similarly, a gradient Ricci soliton \eqref{soliton}  is called
 {\it  steady} and {\it expanding} if the real number $\lambda$ is $0$ and negative respectively.\\


Gradient Ricci solitons  are important in understanding the Hamilton's Ricci 
flow \cite{Ha}  (see \cite{MR2274812} also). They are self-similar solutions of the Ricci flows and arise often as singularity models of the Ricci 
flow. Indeed, after the work of  Perelman \cite{P1} and others, it has been confirmed by Enders,  M\"uller and Topping \cite{Topping} that, under certain mild restriction, the blow-ups around a Type I singularity point of a Ricci flow converge to (nontrivial) gradient shrinking Ricci solitons.     Therefore, it is  natural to seek classification for gradient  Ricci solitons. We refer the readers to  \cite{caoALM11} and references therein for a overview on the subject.  \\

 Shrinking gradient Ricci solitons  are well understood in dimensions $2$ and $3$. It was proved by Hamilton \cite{Ha}  that a two-dimensional complete gradient shrinking Ricci soliton is either isometric to the plane $\Bbb{R}^2$ or a quotient of the sphere $\Bbb{S}^2.$  For dimension $3$, from the works of  Ivey \cite{MR1249376},
Perelman  \cite{P1}, Naber \cite{Naber},  Ni and Wallach \cite{Ni}, and  Cao, Chen and Zhu \cite{CCZ}, it is known that any complete three-dimensional gradient shrinking Ricci soliton is a finite quotient of either the round sphere $\Bbb{S}^3,$ or the Gaussian shrinking soliton $\Bbb{R}^3,$ or the round cylinder $\Bbb{S}^{2}\times\Bbb{R}.$ In higher dimensions, in recent years, there has been  much progress. See, e.g., \cite{MR2448435, Ni, zhang, Naber, MR2732975, PW, CWZ, FLGR, KW, MS, MW2} and the references therein. \\

  In \cite{MR2507581}, Petersen
and Wylie  defined a gradient Ricci soliton $(M, g)$ to be {\it rigid} if it is isometric to a quotient  $N\times_{\Gamma}\mathbb{R}^k$, where $N$ is an Einstein manifold, $\Gamma$ acts freely on $N$ and by orthogonal
transformations on $\mathbb{R}^k$ (no translational components) to get a flat vector bundle over a base that is Einstein and with $f=\frac{\lambda }{2}d^2$, where $d$ is the distance in the flat fibers to the base.
In shrinking case, since  a gradient shrinking Ricci soliton has the finite fundamental group (\cite{Morgan}),   a gradient shrinking Ricci soliton is rigid if it is  isometric to a finite quotient  $N\times_{\Gamma}\mathbb{R}^k$.\\


Different from dimensions $2$ and $3$,  non-rigid compact shrinking (K$\ddot{\text{a}}$hler) gradient solitons were  constructed in dimension 4   by Koiso \cite{Ko} and Cao \cite{Cao1}.  In
any dimension, Eminenti,  La Nave and Mantegazza \cite{MR2448435} proved that a compact shrinking soliton is rigid if and only if  its scalar
curvature is constant.\\

Naber \cite{Naber} showed that any $4$-dimensional non-flat complete noncompact gradient shrinking Ricci soliton
with bounded non-negative curvature operator is a quotient of either $\mathbb{S}^3\times \mathbb{R}$ or $\mathbb{S}^2\times \mathbb{R}^2$. Petersen
and Wylie  (\cite[Theorem 1.2]{MR2507581}) proved that a complete gradient Ricci soliton  is rigid if and only if it has
constant scalar curvature and is radially flat, that is, the sectional curvature $K(\cdot, \nabla f)=0$. They also showed that  the scalar curvature $S$ of a gradient Ricci soliton  is $0$ or $n\lambda$,
 if and only if the underlying Riemannian structure is
Einstein \cite[Proposition 3.3]{MR2507581}.  \\

 Recently, Fern\'{a}ndez-L\'{o}pez and  Garc\'\i a-R\'\i o  \cite{FR} proved the following results for complete $n$-dimensional gradient Ricci solitons \eqref{soliton} with constant scalar curvature $S$:  (i) The possible value of $S$ is  $\{0, \lambda, \cdots, (n-1)\lambda, n\lambda\}$  \cite[Theorem 1]{FR}. (ii) If $S$ takes the value $(n-1)\lambda$,  then the soliton must be rigid. In the shrinking case,  there is  no any complete gradient shrinking Ricci soliton   with $S=\lambda$ \cite[Theorem 10]{FR}.  (iii) Any four-dimensional gradient
shrinking  Ricci soliton with constant scalar curvature $S = 2\lambda$ has
non-negative  Ricci curvature \cite[Theorem 5]{FR}.\\

During the IX Workshop on Differential Geometry (2019) in Maceió, Brazil,  Huai-Dong Cao raised the following 

\noindent {\bf Conjecture}: 
Let $(M^n, g, f)$, $n\geq 4$,  be a complete $n$-dimensional gradient shrinking Ricci soliton. If $(M, g)$ has constant scalar curvature, then it must be rigid, 
i.e., a finite quotient of $\mathbb{N}^k\times \mathbb{R}^{n-k}$ for some Einstein manifold $\mathbb{N}$ of positive scalar curvature." \\

Now we focus on   $4$-dimensional complete shrinking gradient Ricci solitons with constant scalar curvature. By the results in \cite{FR} and \cite{MR2507581} stated above, 
 their  scalar curvature  $S\in \{0,  2\lambda, 3\lambda, 4\lambda\}$.  Moreover, if $S=0$ or $ 4\lambda$, they are   Einstein;   If $S=3\lambda$,  they are  finite quotient of $\Bbb{S}^{3}\times\Bbb{R}$. Therefore, $S=2\lambda$ is the unique unknown case.\\
 
 In this paper we solve this case and prove the following result.
 
\begin{theorem}\label{thm-rig} 
 Let $(M, g, f)$
be a $4$-dimensional  complete noncompact gradient shrinking Ricci soliton satisfying \eqref{soliton}. If $M$  has the constant scalar curvature $S=2\lambda$, then 
it must be isometric to a finite quotient of $\mathbb{S}^2\times \mathbb{R}^2$. 
\end{theorem}

In dimension $4$,  a  complete shrinking gradient Ricci soliton is   rigid if it is  isometric to either Einstein, or a finite quotient of Gaussian shrinker $\Bbb{R}^4$, $\Bbb{S}^{2}\times\Bbb{R}^{2}$ or $\Bbb{S}^{3}\times\Bbb{R}.$ Therefore,
Theorem \ref{thm-rig} together with \cite[Proposition 3.3]{MR2507581} and  \cite[Theorem 5]{FR}  confirms Cao's conjecture in dimension $n=4$.

\begin{theorem}\label{thm-rig-1} 
 Let $(M, g, f)$
be a $4$-dimensional  complete noncompact gradient shrinking Ricci soliton. Then  $M$ has the constant scalar curvature  if and only if it is rigid, that is,  it is  isometric to either Einstein, or a finite quotient of Gaussian shrinking soliton $\Bbb{R}^4$, $\Bbb{S}^{2}\times\Bbb{R}^{2}$ or $\Bbb{S}^{3}\times\Bbb{R}$.
\end{theorem}

In order to prove Theorem \ref{thm-rig}, we calculate the drifted Laplacian $\Delta_f\textrm{tr}(\text{Ric}^3)$ of the trace of the tensor $\textrm{Ric}^3$ for $4$-dimensional  gradient shrinking Ricci soliton with constant scalar curvature $2\lambda$ and then prove Inequality \eqref{tr-ine}, which is a drifted Laplacian  inequality acting on the nonnegative function $f(\textrm{tr}(\text{Ric}^3)-\frac14)$ with constant coefficients. Using this inequality together with some geometric properties of $4$-dimensional  shrinking solitons such as the curvature estimates of Munteanu and Wang \cite{MW15}  and volume estimates of Cao and the second author \cite{MR2732975},  we consequently get a way to prove Theorem  \ref{thm-rig}. \\

In the proof of  Inequality \eqref{tr-ine},  the isoparametric property of the potential function $f$ plays a very important role. Based on this property, the curvature tensor    of $M$ is related with  the curvature tensors of the  level sets of $f$ by  Gauss equations. One of the crucial facts is that the level sets of $f$ is three-dimensional, and its curvature tensor  can be expressed by its Ricci tensor. It is also worth mentioning that in the calculation of the drifted Laplacian of the trace of the tensor $\textrm{Ric}^3$, the trace of the tensor $\textrm{Ric}^4$ appears. However  it can be expressed  by the trace of the tensor $\textrm{Ric}^3$ using the algebraic formula.\\


In general, the classification of  gradient shrinking Ricci solitons  is more subtle in dimensions more than three, mainly due to the non-triviality of
the Weyl tensor W.  There are some results on the  four-dimensional case under some vanishing conditions involving the Weyl tensor $W.$  By the works of \cite{MR2448435, Ni, zhang, PW, CWZ, MS}, it is known that complete locally conformally flat (i.e. $W=0$) four-dimensional gradient shrinking Ricci solitons are isometric to finite quotients of either $\Bbb{S}^4$, $\Bbb{R}^4$, or $\Bbb{S}^{3}\times\Bbb{R}.$ Recently, X. Chen and Y. Wang \cite{CW} proved that half-conformally flat, that is, $W^{+}=0$ or $W^{-}=0$, four-dimensional gradient shrinking Ricci solitons are isometric to finite quotients of $\Bbb{S}^4$, $\Bbb{CP}^2,$ $\Bbb{R}^4$, or $\Bbb{S}^{3}\times\Bbb{R}.$ Then, X. Cao and Tran \cite{CH} proved that  the condition  $W^+(\nabla f, \cdot, \cdot, \cdot)=0$ for a complete  four-dimensional gradient shrinking Ricci soliton can imply it is either Einstein or has $W^+=0$. In \cite{CaoChen}, Cao and Q. Chen  showed that complete Bach-flat four-dimensional gradient shrinking Ricci solitons are either Einstein or finite quotients of $\Bbb{R}^4$ or $\Bbb{S}^{3}\times\Bbb{R}$. It is known  that either locally conformally flat metrics, Einstein metrics, or half-conformally 
flat metrics on four-dimensional manifolds are Bach-flat.\\

On the other hand, it was  showed, by Fern\'{a}ndez-L\'opez and Garc\'ia-R\'io \cite{FLGR} in compact case, and  Munteanu and Sesum \cite{MS} in noncompact case, that complete gradient shrinking Ricci solitons with harmonic Weyl tensor (i.e. $\delta W=0$) are rigid, namely, they are either Einstein, or a finite quotient of $\Bbb{R}^4$, $\Bbb{S}^{2}\times\Bbb{R}^{2}$ or $\Bbb{S}^{3}\times\Bbb{R}.$ Recently, J. Wu, P. Wu and Wylie \cite{Wu} proved  that the same conclusion holds under the weaker condition of harmonic self-dual Weyl tensor (i.e. $\delta W^{+}=0$).   We would like to point out that  Kim \cite{MR3625140} proved that the condition of  local  harmonicity  of Weyl tensor implies local classification  of four dimensional gradient shrinking Ricci solitons. This condition implies that     the scalar curvature is constant.  There are also other classification results under the various pinching conditions involving the Weyl tensor $W$ (see  \cite{catinoAdv},  \cite{CH}, \cite{Catino},  \cite{Wu}, \cite{Zhang2}, \cite{CRZ}). \\

\begin{remark} 
In a forthcoming paper (\cite{CZ}), we discuss $4$-dimensional gradient expanding Ricci solitons with constant scalar curvature and get some similar features.  Due to the infinity of their weighted volume, the same results are not expected.
\end{remark}

\noindent {\bf Acknowledgment.}   The authors would like to thank Professor Huai-Dong Cao for his helpful suggestions.

\section{Gradient shrinking Ricci solitons}

\subsection{Notations and basic formulas}

Let  $(M, g)$   be an $n$-dimensional complete gradient  shrinking Ricci soliton with the potential function $f$, that is, it holds that  on $M$,
\begin{align}\label{soliton-1}
\text{Ric}+\nabla^2f=\lambda g
\end{align}
 for some  positive  real number $\lambda>0$.\\

By
scaling the metric $g$,  one can normalize $\lambda$ so that $\lambda=\frac{1}{2}$, that is,
\begin{align}\label{soliton-2}
\text{Ric}+\nabla^2f=\frac{1}{2} g.
\end{align}
 Tracing the soliton equation \eqref{soliton-2} we get 
 \begin{align}\label{prop1}
 \Delta f+S=\frac{n}{2},
 \end{align}
 where $S$ denotes the scalar curvature of $M$.\\
 
It is well known (see \cite{Ha1}) that  $S+|\nabla f|^2-f $ is constant and further   $f$ can be translated by a constant so that 
\begin{align}\label{prop2}
S+|\nabla f|^2=f. 
\end{align}

Also, by a result of  Chen \cite{chen}, the scalar curvature $S>0$ unless $
M$ is flat. So in the following, we will assume without loss of generality
that $S>0.$\\

For $M$, the drifted Laplacian  $\Delta _{f}=\Delta -\left\langle \nabla f,\nabla \right\rangle $ acting on the functions 
 is self-adjoint on the space of square
integrable functions with respect to the weighted measure $e^{-f}dv.$ In general, the drifted Laplacian $\Delta_f$ acting on tensors is given by $\Delta_f=\Delta-\nabla_{\nabla f}$.\\

Concerning the potential function $f,$ Cao and Zhou \cite{MR2732975} have proved
that

\begin{equation}\label{f}
\left( \frac{1}{2}r(x) -c\right) ^{2}\leq f(x) \leq
\left( \frac{1}{2}r(x) +c\right) ^{2} 
\end{equation}
for all $r(x) \geq r_{0}.$ Here $r(x):=d\left(
p,x\right)$ is the distance of $x$ to $p,$ a minimum point of $f$ on $M,$
which always exists. Both constants $r_{0}$ and $c$ can be chosen to depend
only on dimension $n.$\\

Throughout the paper, we denote 
\begin{eqnarray*}
D\left( t\right) &=&\left\{ x\in M:f(x)\leq \frac{t^2}4\right\} \\
\Sigma \left( t\right) &=&\partial D\left( t\right) =\left\{ x\in M:f(x) =\frac{t^2}4\right\} .
\end{eqnarray*}%
By (\ref{f}), these are compact subsets of $M$.\\

In this paper,  our  sign convention for curvatures is as follows:

$$\text{Rm}(X,Y)=\nabla^2_{Y,X}-\nabla^2_{X,Y}, \quad  \text{Rm}(X,Y,Z,W)=g(\text{Rm}(X,Y)Z,W),$$
$$ K(e_i,e_j)=\text{Rm}(e_i,e_j,e_i,e_j), \quad  \text{Ric}(X,Y)=\text{tr} \text{Rm}(X, \cdot, Y, \cdot), \quad S=\text{tr}\text{Ric}.$$
\vskip 0.2cm

We  recall the following equations for curvatures. For their proofs, one may
consult \cite{PW}. 

\begin{eqnarray}
\nabla _{l}R_{ijkl} &=&R_{ijkl}f_{l}=\nabla _{j}R_{ik}-\nabla _{i}R_{jk} \label{eq1}
\\
\nabla _{j}R_{ij} &=&R_{ij}f_{j}=\frac{1}{2}\nabla _{i}S  \label{eq2}\\
-R_{ikjl}f_lf_k&=&\frac12\nabla_i\nabla_j S+(\nabla_kf_{ij})f_k+R_{kj}(R_{ik}-\frac12 g_{ik}) \label{eq3} \\
 \Delta _{f}\mathrm{Rm} &=&\mathrm{Rm}+\mathrm{Rm}\ast \mathrm{Rm}  \label{eq4}\\
 \Delta _{f}R_{ij} &=&R_{ij}-2R_{ikjl}R_{kl}  \label{eq6}\\
\Delta _{f}S &=&S-2\left\vert \mathrm{Ric}\right\vert ^{2} =\langle\mathrm{Ric},g-2\mathrm{Ric}\rangle \label{id}.\label{eq5}
\end{eqnarray}

\eqref{eq3} was proved in   \cite[Proposition 1]{PW}. Since it is very important for the proof of Theorem \ref{thm-rig}, we give a proof here for the sake of completeness.

\begin{proof}
By \eqref{eq2}: $\nabla_jS=2R_{kj}f_k$, we have 
\begin{align*}
\nabla_i\nabla_jS&=2 (\nabla_iR_{kj})f_k+2R_{kj}(\nabla_i\nabla_kf)\\
&=2(\nabla_kR_{ij})f_k-2R_{ikjl}f_lf_k+2R_{kj}(\nabla_i\nabla_kf)\\
&=-2(\nabla_kf_{ij})f_k-2R_{ikjl}f_lf_k+2R_{kj}(\nabla_i\nabla_kf)\\
&= -2(\nabla_kf_{ij})f_k-2R_{ikjl}f_lf_k+2R_{kj}(\frac12 g_{ik}-R_{ik}).
\end{align*}
Then,
\begin{align*}
-2R_{ikjl}f_lf_k=\nabla_i\nabla_jS+2(\nabla_kf_{ij})f_k+2R_{kj}(R_{ik}-\frac12 g_{ik}).
\end{align*}
Multiplying $-\frac12$ on both sides of the above equality yields \eqref{eq3}.

\end{proof}

If the scalar curvature $S$ is constant, \eqref{eq5} implies 
\begin{align}
|\text{Ric}|^2&=\frac{S}{2},
\end{align}
and  \eqref{eq2} yields 
\begin{align}\label{ric4}
\text{Ric} (\nabla f, \cdot)=0.
\end{align}

\subsection{Isoparametric functions} 
For a general Riemannian manifold $M$, a non-constant $C^2$ function $f: M\rightarrow \mathbb{R}$ is called 
\emph{transnormal}  if 
\[ |\nabla f|^2=b(f) \]
for som $C^2$ function $b$ on the range of $f$ in $\mathbb{R}$. The function $ f$  is called
\emph{isoparametric} if, in addition,  it satisfies
\[\Delta f=a(f)\]
for some continous function $a$ on the range of $f$ in $\mathbb{R}$ (see \cite{Bol}, \cite{GT}, \cite{Mi},\cite{MR901710}).\\

 For a general Riemannian manifold $M$,  the preimage of the  minimum or maximum of a transnormal function $f$ is  the
\emph{focal variety} of $f$, denoted by $M_{-}$  or $M_{+}$.  They could be empty sets.
A fundamental structural result given in \cite{MR901710} says that the
focal varieties  of a transnormal function on a complete Riemannian manifold are smooth
submanifolds and each regular level hypersurface $\Sigma(t):=f^{-1}(\frac{t^2}4)$, $t>0$, is a tubular hypersurface
over   either of $M_{-}$ and $M_{+}$. 
Moreover, if $f$ is isoparametric,  the following result holds (see  \cite{MR901710}
and \cite{GT} for a proof):
\begin{theorem}\cite[Theorem D]{MR901710}\label{minimal focal}
The focal varieties of an isoparametric function $f$ on a
complete Riemannian manifold  are smooth minimal submanifolds.
\end{theorem}

Now we come back to the case of complete gradient shrinking Ricci solitons with constant scalar curvature $S$. The  potential function $f$  can be renormalized, by replacing $f-S$  with $f$, so that $f:M\to [0, +\infty)$,  

\begin{equation}\label{iso1}
|\nabla f|^2=f,
\end{equation}
and
\begin{equation}\label{iso2}
\triangle f=\frac{n}2-S.
\end{equation}
Therefore  the (nonconstant) renormalized $f$ is an isoparametric function on $M$. From the potential function estimate \eqref{f},   $f$ is proper and  is unbounded. 
 By the theory of isoparametric functions, these properties together with 
(\ref{iso1}) imply that the minimum $0$ of $f$  is the unique singular value of $f$  and  the regular level hypersurfaces $\Sigma(t)=f^{-1}(\frac{t^2}4)$, $t>0$, are parallel. Moreover \eqref{iso2} implies that  these hypersurfaces have constant mean
curvatures. These  regular level hypersurfaces $\Sigma(t)$, $t>0,$  are called \emph{isoparametric
hypersurfaces}. \\

In the following, we derive some facts for complete gradient shrinking Ricci solitons with constant scalar curvature $S$. 
\begin{theorem}\label{thm-cscs} 
 Let $(M, g, f)$
be an $n$-dimensional  complete noncompact gradient shrinking Ricci soliton satisfying \eqref{soliton-2} with constant scalar curvature $S$ and let $f$ be normalized as 
\begin{align}|\nabla f|^2=f. 
\end{align}
 Then (i) $M_{-}=f^{-1}(0)$ is a compact and  connected minimal submanifold of $M$. 
 
 (ii) The function $f$ can be expressed as $$f(x)=\frac14 \dist^2 (x, M_{-}).$$
 
 (iii) For any point $p\in M_{-}$, $\mathrm{Hess}(f)$ has two eigenspaces $T_pM_{-}$ and $\nu_pM_{-}$ corresponding eigenvalues $0$ and $\frac12$, and $\dim(M_{-})=2S$.\\
 
 (iv) Let $D_t:=\{x\in M; \dist (x, M_{-})<t\},$ for $t>0$. Then  mean curvature $H(t)$ of the smooth hypersurface $\Sigma(t)=\partial D_t$ satisfies
\[H(t)=\frac{n-2S-1}t.\]
 
 (v) The volume of  the set $D_t$ satisfies  
 \[\textrm{vol} (D_t)= \frac{1}{k}|M_{-}|\omega_k t^k,\]
  where  $k=n-2S$, $|M_{-}|$ denotes the volume of the submanifold $M_{-}$, and  $\omega_{k-1}$ is the area of the unit sphere in $\mathbb{R}^k$.
\end{theorem} 

\begin{proof} (i) From Theorem D in \cite{MR901710}, $\Sigma$ is a smooth submanifold. 
Since the function $f$ is proper, $M_{-}=f^{-1}(0)$ is compact. The connectedness of $M_{-}$ follows from the fact that $0$ is the unique critical value of $f$.\\

  To prove (ii), we continue to use transnormality of the potential function $f$. The level sets of $f$ form a foliation of transnormal system and distance between the level sets are realized by integrabale curves of $\nabla f$ which are orthogonal to every leaf.  Therefore, for any  $0<a<b$, there is a normalized geodesic $\gamma: [0,l]\to M$ realizing the distance $\dist\big(f^{-1}(a),f^{-1}(b)\big)$ such that $\gamma'\parallel \nabla f$, and 
\[
\begin{split}
\dist\big(f^{-1}(a), f^{-1}(b)\big)&=\int_0^l|\gamma'(t)|dt\cr
&=\int_0^l\frac{|(f\circ \gamma)'(t)|}{|\nabla f|}dt\cr
&=\int_a^b\frac{df}{|\nabla f|}=\int_a^b\frac{df}{\sqrt{f}}=2\sqrt{b}-2\sqrt{a}.
\end{split}
\]
Letting  $a\to 0$, we get the conclusion in (ii).\\

The general case of (iii)  was proved in \cite{FR}. In fact, it follows from Lemma 6 in \cite{MR901710}.\\

(iv) Since $S$ is constant,  $0=\nabla S=2\text{Ric}(\nabla f)$ and $\textrm{Hess} f(\nabla f,\nabla f)=\frac12|\nabla f|^2.$ Since $\partial D_t=f^{-1}(\frac{t^2}4)$, the mean curvature $H(t)$ is 
\[
\begin{split}
H(t)&=\frac{\Delta f-\textrm{Hess} f(\frac{\nabla f}{|\nabla f|},\frac{\nabla f}{|\nabla f|})}{|\nabla f|}\\
&=\frac{\frac{n}2-S-\frac12}{\frac{t}2}\\
&=\frac{n-2S-1}{t}
\end{split}
\]
(v)
Denote $V(t)=\textrm{vol}(D_t)$ and $k=n-2S$. Then the co-area formula implies $V'(t)=\int_{\partial D_t}$ and the first variation formula implies
\[V''(t)=\int_{\partial D_t}H(t)=\frac{k-1}{t}V'(t).
\]
Therefore $t^{-k+1}V'(t)$ is a constant. On the other hand,  by the area  formula of the tubular hypersurfaces, we have  
\begin{align}\label{tubevol}
\displaystyle t^{-k+1}V'(t)=\int_{M_{-}}\int_{\mathbb{S}^{k-1}(1)}\nu_u(p, t)dudv^{M_{-}},
\end{align}
where $\nu_u(p, t)$ denotes the infinitesimal change of the volume function of $M_{-}$ in the direction $u\in T_p^{\perp}M_{-}, |u|=1, p\in M_{-}$.  Letting $t\rightarrow 0$ in \eqref{tubevol}, we get that $t^{-k+1}V'(t)$
 is equal to $|M_{-}|\omega_{k-1}$. This implies that $V(t)= \frac{1}{k}|M_{-}|\omega_{k-1} t^k$.

\end{proof}

\section{Rigidity of $4$-dimensional gradient shrinking Ricci solitons}

In this section, the gradient shrinking Ricci soliton $(M, g, f)$ is assumed to have constant scalar curvature.
For convenience, we choose $\lambda=\frac12$ and the normalized potential function $f$ so that $(M, g, f)$ satisfies 
\begin{align}\label{soliton-3}
\text{Ric}+\nabla^2f=\frac{1}{2} g
\end{align}
and
 \begin{align}\label{f-norm}
|\nabla f|^2=f. 
\end{align}

\eqref{soliton-3} implies that
\begin{align}\label{delta}
 \Delta f+S=\frac{n}{2},
 \end{align}
 where $S$ denotes the scalar curvature of $M$.\\
 
  Now we consider the case that $S=2\lambda=1$ with motivation stated in the Introduction. We first calculate the drifted Laplacian of the trace of $\text{Ric}^3$ in this case and obtain the following results:

 
\begin{proposition}\label{tr}	
Let  $(M^4, g, f)$ be a $4$-dimensional complete normalized gradient shrinking Ricci soliton satisfying \eqref{soliton-3} and \eqref{f-norm} with the scalar curvature $S=1$. Then
\begin{itemize}
\item [ ] (i) $\displaystyle\Delta_f\textrm{tr}(\text{Ric}^3)=4\textrm{tr}(\textrm{Ric}^3)-1-\frac32|\nabla Ric|^2
   +6(\nabla \text{Ric})\ast (\nabla \text{Ric}) \ast \text{Ric}.$\\
   
\item [ ] (ii)  $\Delta_f\textrm{tr}(Ric^3)\geq 4\textrm{tr}(\textrm{Ric}^3)-1-\dfrac32|\nabla Ric|^2.$
\end{itemize}
Here,    $(\nabla \text{Ric})\ast (\nabla \text{Ric}) \ast \text{Ric}=(\nabla_sR_{ij})(\nabla_sR_{jk})R_{ki}$ under local orthonormal frame.

\end{proposition}

\begin{proof}
We firstly  calculate the drifted Laplacian $\Delta_f$ of the  tensor $\textrm{Ric}^3$. Note that $\Delta_f=\Delta-\nabla_{\nabla f}$. Under a local orthonormal frame $\{e_1, e_2, e_3, e_4\}$, a direct computation gives
\begin{align*}
\Delta_f& (\textrm{Ric}^3)_{il}\\
=&(\Delta_f R_{ij})R_{jk}R_{kl}+ R_{ij}(\Delta_f R_{jk})R_{kl}+R_{ij}R_{jk}(\Delta_f R_{kl})\\
   &+2(\nabla_sR_{ij})(\nabla_sR_{jk})R_{kl}+2(\nabla_sR_{ij})R_{jk}(\nabla_sR_{kl}) +2R_{ij}(\nabla_sR_{jk})(\nabla_sR_{kl})\\
   =&3(\textrm{Ric}^3)_{il}-2(R_{isjt}R_{st}R_{jk}R_{kl}+R_{jskt}R_{st}R_{ij}R_{kl}+R_{kslt}R_{st}R_{ij}R_{jk})\\
   &+2(\nabla_sR_{ij})(\nabla_sR_{jk})R_{kl}+2(\nabla_sR_{ij})R_{jk}(\nabla_sR_{kl}) +2R_{ij}(\nabla_sR_{jk})(\nabla_sR_{kl}).
\end{align*}
In the last equality, we have used  \eqref{eq6}.
Then,  the drifted Laplacian of the trace of $\text{Ric}^3$ is as follows:
\begin{align}
\Delta_f \big(\textrm{tr}(\textrm{Ric}^3)\big)&=3\textrm{tr}(\textrm{Ric}^3)+2(\nabla_sR_{ij})(\nabla_sR_{jk})R_{ki}\nonumber\\
   &\quad +2(\nabla_sR_{ij})R_{jk}(\nabla_sR_{ki}) +2R_{ij}(\nabla_sR_{jk})(\nabla_sR_{ki})\nonumber\\
   &\qquad-2(R_{isjt}R_{st}R_{jk}R_{ki}+R_{jskt}R_{st}R_{ij}R_{ki}+R_{ksit}R_{st}R_{ij}R_{jk})\nonumber\\
   &=3\textrm{tr}(\textrm{Ric}^3)-6R_{isjt}R_{st}R_{jk}R_{ki}
   +6(\nabla_sR_{ij})(\nabla_sR_{jk})R_{ki}. \label{eq-trace}
\end{align}

Fix  $p\in M$ and assume $R_{ij}=a_i\delta_{ij}$ at $p$, that is, $a_i$, $i=1, \cdots, 4$, are the eigenvalues of the Ricci tensor at $p$. Hence,
\[R_{isjt}R_{st}R_{jk}R_{ki}=\sum_{i=1}^4\sum_{s=1}^4R_{isis}a_sa_i^2.
\]
Since $S$ is constant on $M$, $\text{Ric}(\nabla f, \cdot)=0$ and  $\nabla f$ is the eigenvector of Ricci curvature corresponding to the eigenvalue $0$. We can arrange  the eigenvalues of Ricci curvature  so that $a_4=0$ is the corresponding eigenvalue of $\nabla f$. By this manner, we have
\[a_1+a_2+a_3=S=1\]
and 
\begin{align}\label{term-curvature}
R_{isjt}R_{st}R_{jk}R_{ki}=\sum_{i=1}^3\sum_{s=1}^3R_{isis}a_s a_{i}^2.
\end{align}
Denote  the sectional curvature $K(e_i, e_s)=R_{isis}$ by $K_{is}$. 
We get
\begin{align}
R_{isjt}R_{st}&R_{jk}R_{ki}\nonumber\\
&=K_{12}a_1a_2^2+K_{13}a_1a_3^2+K_{23}a_2a_3^2+K_{12}a_1^2a_2+K_{13}a_1^2a_3+K_{23}a_2^2a_3\nonumber\\
&=K_{12}a_1a_2(a_1+a_2)+K_{13}a_1a_3(a_1+a_3)+K_{23}a_2a_3(a_2+a_3)\nonumber
\\
&=K_{12}a_1a_2(1-a_3)+K_{13}a_1a_3(1-a_2)+K_{23}a_2a_3(1-a_1)\nonumber
\\
&=K_{12}a_1a_2+K_{13}a_1a_3+K_{23}a_2a_3-(K_{12}+K_{13}+K_{23})a_1a_2a_3\nonumber\\
&=K_{12}a_1a_2+K_{13}a_1a_3+K_{23}a_2a_3-\frac12 a_1a_2a_3.\label{trace-1}
\end{align}

\noindent Here we have used $a_1+a_2+a_3=1$  in the third equality in \eqref{trace-1} and $ K_{12}+K_{13}+K_{23}=\frac{S}2=\frac12$  in the fourth equality by noting that  $K_{14}+K_{24}+K_{34}=R_{44}=a_{4}=0$.\\

\noindent Substituting \eqref{trace-1} into \eqref{eq-trace} yields 
\begin{align}
\Delta_f (\textrm{tr}(\textrm{Ric}^3&))\nonumber\\
   =&3\textrm{tr}(\textrm{Ric}^3)-6(K_{12}a_1a_2+K_{13}a_1a_3+K_{23}a_2a_3-\frac12a_1a_2a_3)\nonumber\\
  &\quad +6\sum_{i,s=1}^3\sum_{j=1}^4(\nabla_sR_{ij})(\nabla_sR_{ji})a_i. \label{trace-2} 
\end{align}
By \eqref{eq6}, it holds that on $M$,
\begin{align}\label{ric-square}
\Delta_f|\text{Ric}|^2=2|\text{Ric}|^2+2|\nabla \text{Ric}|^2-4R_{ij}R_{ikjl}R_{kl}.
\end{align}
Since  \eqref{eq5} implies $|\text{Ric}|^2=\frac{S}{2}=\frac12$ is constant,  \eqref{ric-square} becomes
\begin{align}\label{ric-square-1}1+2|\nabla \text{Ric}|^2-4R_{ij}R_{ikjl}R_{kl}=0.
\end{align}
Note that  $R_{ij}=a_i\delta_{ij}, i,j=1, \cdots, 4$, and $a_4=0$ at $p$.  \eqref{ric-square-1} yields
\begin{align}\label{ric-square-2}
K_{12}a_1a_2+K_{13}a_1a_3+K_{23}a_2a_3=\frac18+\frac14|\nabla\text{Ric}|^2.
\end{align}
On the other hand, 
  letting $a=a_1, b=a_2$, and $c=a_3$ in the  following algebraic identity:
\[
(a+b+c)^3=3(a+b+c)(a^2+b^2+c^2)-2(a^3+b^3+c^3)+6abc,
\]
and noting that $\sum_{i=1}^3a_i=S=1$ and $\sum_{i=1}^3a_i^2=|\text{Ric}|^2=\frac12$, we have
\begin{align}\label{product}
a_1a_2a_3=\frac13\textrm{tr}(\text{Ric}^3)-\frac1{12}.
\end{align}
Plugging  \eqref{product} and \eqref{ric-square-2} into \eqref{trace-2}, we get
\begin{align}\label{tr-3}
\Delta_f (\textrm{tr}(\textrm{Ric}^3))&\nonumber\\
  =&3\textrm{tr}(\textrm{Ric}^3)-6(\frac18+\frac14|\nabla \textrm{Ric}|^2)+3(\frac13\textrm{tr}(\textrm{Ric}^3)-\frac1{12})\nonumber\\
   &\quad +6\sum_{i,s=1}^3(\nabla_sR_{ij})(\nabla_sR_{ji})a_i \nonumber  \\
   =&4\textrm{tr}(\textrm{Ric}^3)-1-\frac32|\nabla \text{Ric}|^2
   +6\sum_{i, s=1}^3\sum_{j=1}^4(\nabla_sR_{ij})^2a_i.
\end{align}
Hence we get  (i).  \\

It was proved in \cite[Proposition 12]{FR} that for a complete gradient shrinking  Ricci soliton
with constant scalar curvature $S = k\lambda$,  if the rank of Ricci curvature is  no more than  $k+1$, then the Ricci curvature
is non-negative. In particular, a four-dimensional complete  gradient shrinking  Ricci soliton
with  $S = 2\lambda$ has  nonnegative Ricci curvature \cite[Theorem 5]{FR}. Here we  give a simple proof   for this special case.\\

Otherwise, suppose that  one of eigenvalues $a_1, a_2, a_3$ is negative, say $a_3<0$.
Noting that $a_1+a_2+a_3=1$, we have $a_1+a_2>1$ and $a_1^2+a_2^2\geq \frac{(a_1+a_2)^2}2>\frac12,$ which is a contradiction with  $a_1^2+a_2^2+a_3^2=\frac12$.
Then $a_i\geq 0$, $i=1, 2, 3$. This gives that  the Ricci curvature is nonnegative.\\

By $a_i\geq 0$, $i=1, 2, 3$,  the last  term in the right side of \eqref{tr-3} is nonnegative. Hence we get 
     \[
\Delta_f\textrm{tr}(\text{Ric}^3)\geq 4\textrm{tr}(\textrm{Ric}^3)-1
 -\frac32|\nabla \text{Ric}|^2,
\]
which is (ii).

\end{proof}
Now we prove an inequality, which plays a key role in the proof of Theorem \ref{thm-rig}.

\begin{lemma}	\label{lemma2}
Let  $(M^4, g, f)$ be a $4$-dimensional complete normalized gradient shrinking Ricci soliton satisfying \eqref{soliton-3} and \eqref{f-norm} with the scalar curvature $S=1$. Then
\begin{align}\label{tr-ine}
\Delta_f\big[ f\big(\textrm{tr}(\text{Ric}^3)-\frac14\big)\big]&\geq 9 f\big(\textrm{tr}(\text{Ric}^3)-\frac14\big).
\end{align}

\end{lemma}
\begin{proof}	
Let  $\Sigma(t)= f^{-1}(\frac{t^2}4),  t>0,$ be the level sets of $f$. It is known that each $\Sigma(t)$   is a smooth hypersurface. From now on, since  $t$ is fixed,  we omit $t$ in the notation $\Sigma(t)$ for simplicity. Choose a  local orthonormal frame $\{ e_1,e_2, e_3\}$ for $\Sigma$ so that $\{ e_1,e_2, e_3, e_4\}$, where $e_4=\frac{\nabla f}{|\nabla f|}$, is a   local orthonormal frame for $M$. Recall that the  intrinsic curvature tensor  $R^\Sigma_{\alpha\beta\gamma\eta}$ and the extrinsic curvature tensor $R_{\alpha\beta\gamma\eta}$ of $\Sigma$, where $\alpha, \beta, \gamma, \eta\in\{1,2,3\}$,  are related  by  the Gauss equations: 
\begin{align}\label{Gauss}
R^\Sigma_{\alpha\beta\gamma\eta}=R_{\alpha\beta\gamma\eta}+h_{\alpha\gamma}h_{\beta\eta}-h_{\alpha\eta}h_{\beta\gamma}
\end{align}
where $h_{\alpha\beta}$ denote the components of  the second fundamental form $A$ of $\Sigma$.  Moreover,
\begin{align}\label{GaussRic}
R^\Sigma_{\alpha\gamma}=R_{\alpha\gamma}-R_{\alpha4\gamma4}+Hh_{\alpha\gamma}-h_{\alpha\beta}h_{\beta\gamma},
\end{align}
and the scalar curvature $S^\Sigma$  of $\Sigma$ satisfies
\begin{align}\label{GaussS}S^\Sigma=S-2R_{44}+H^2-|A|^2.
\end{align}
In what follows,  the indices  $\alpha, \beta, \gamma, \eta\in\{1,2,3\}$ and $i, j, k\in \{1,2,3, 4\}$.\\

Since  $\textrm{Ric}(\nabla f, \cdot)=\frac12\nabla S(\cdot)=0$, $R_{4i}=0$, $i=1, \ldots, 4$. Then
\begin{align}\label{scalar}
S^\Sigma&=S+H^2-|A|^2.
\end{align}
Noting $\Sigma=f^{-1}(t)$ and using \eqref{soliton-3} and \eqref{f-norm}, we have
\begin{align}\label{secondform}
h_{\alpha\beta}=\frac{f_{\alpha\beta}}{|\nabla f|}=\frac{\frac12\delta_{\alpha\beta}-R_{\alpha\beta}}{\sqrt{ f}}.
\end{align}
Then the  mean curvature of $\Sigma$ satisfies
\begin{align}\label{meancurvature}
H=\frac{\frac32-\sum_{\alpha=1}^3R_{\alpha\alpha}}{\sqrt{ f}}=\frac{\frac32-S}{\sqrt{ f}}=\frac{1}{2\sqrt{ f}}.
\end{align}
Fixing $p\in \Sigma$, we can choose the local orthonormal frame $\{ e_1,e_2, e_3\}$  so that $R_{\alpha\beta}=a_{\alpha}\delta_{\alpha\beta}$ at $p$. Denote $R_{44}$  by $a_4$ and observe that $a_4=0$. From \eqref{secondform}, we get
\begin{align}\label{second}
h_{\alpha\beta}=\frac{\frac12-a_{\alpha}}{\sqrt{ f}}\delta_{\alpha\beta}.
\end{align}
Hence \eqref{scalar}, \eqref{meancurvature} and \eqref{second}  imply that
\begin{align}\label{scal}
S^\Sigma&
=1+\frac{1}{ 4f}-\frac1{ f}\sum_{\alpha=1}^3(\frac12-a_{\alpha})^2\nonumber\\
&=1+\frac{1}{4f}-\frac1{ f}(\frac34-\sum_{a=1}^3a_{\alpha}+\sum_{\alpha=1}^3a_{\alpha}^2)\nonumber\\
&=1+\frac{1}{4f}-\frac1{ f}(\frac34- S+|\text{Ric}|^2)\nonumber\\
&=1+\frac{1}{4f}-\frac{3}{4f}+\frac{1}{f}-\frac{1}{2f}=1. 
\end{align}

\noindent In the above,  we have used $a_4=0$, $S=1$ and $|\text{Ric}|^2=\dfrac12$.\\

On the other hand, since the dimension of $\Sigma $ is three, it is known that  the curvature tensor can be expressed in terms of Ricci curvature as follows:
\begin{align*}
R^\Sigma_{\alpha\beta\gamma\eta}=(R^\Sigma_{\alpha\gamma}g_{\beta\eta}-R^\Sigma_{\alpha\eta}g_{\beta\gamma}+R^\Sigma_{\beta\eta}g_{\alpha\gamma}-R^\Sigma_{\beta\gamma}g_{\alpha\eta})-\frac{S^\Sigma}{2}(g_{\alpha\gamma}g_{\beta\eta}-g_{\alpha\eta}g_{\beta\gamma}).
\end{align*}
Hence,  under the local orthonormal frame $\{ e_1,e_2, e_3\}$ at $p$ chosen in the above, for fixed $\alpha, \beta$, $\alpha\neq \beta$, the above equations together with \eqref{GaussRic} imply that the components $R^\Sigma_{\alpha\beta\alpha\beta}$ of curvature tensor of $\Sigma$ are as follows:
\begin{align}\label{coef}
R^\Sigma_{\alpha\beta\alpha\beta}&=R^\Sigma_{\alpha\alpha}+R^\Sigma_{\beta\beta}-\frac{S^\Sigma}{2}\nonumber\\
&=R_{\alpha\alpha}-R_{\alpha4\alpha4}+Hh_{\alpha\alpha}-h_{\alpha\alpha}^2+R_{\beta\beta}-R_{\beta4\beta4}+Hh_{\beta\beta}-h_{\beta\beta}^2-\frac12\nonumber\\
&=a_\alpha+a_{\beta}-R_{\alpha4\alpha4}-R_{\beta4\beta4}+H(h_{\alpha\alpha}+h_{\beta\beta})-h_{\alpha\alpha}^2-h_{\beta\beta}^2-\frac12.
\end{align}
In the above, we have used \eqref{scal} and $h_{\alpha\beta}= 0$, $\alpha\neq \beta$ (see \eqref{second}). Therefore, for fixed $\alpha\neq \beta$ again, 
 using  the Gauss equations \eqref{Gauss} together with \eqref{coef}, \eqref{meancurvature} and \eqref{second}, we get
\begin{align}
R_{\alpha\beta\alpha\beta}&=R^\Sigma_{\alpha\beta\alpha\beta}-h_{\alpha\alpha}h_{\beta\beta}+h_{\alpha\beta}^2\nonumber\\
&=a_\alpha+a_\beta-R_{\alpha4\alpha4}-R_{\beta4\beta4}+H(h_{\alpha\alpha}+h_{\beta\beta})\nonumber\\
&\quad-h_{\alpha\alpha}^2-h_{\beta\beta}^2-\frac12-h_{\alpha\alpha}h_{\beta\beta}\nonumber\\
&=a_\alpha+a_\beta-R_{\alpha4\alpha4}-R_{\beta4\beta4}-\frac12+\frac{1-a_\alpha-a_\beta}{2f} \nonumber\nonumber\\
&\quad -\frac{(\frac12-a_\alpha)^2+(\frac12-a_\beta)^2+(\frac12-a_\alpha)(\frac12-a_\beta)}{ f} \nonumber\\
&=a_\alpha+a_\beta-R_{\alpha4\alpha4}-R_{\beta4\beta4}-\frac12+\frac{-\frac14-a_\alpha^2-a_\beta^2-a_\alpha a_{\beta}+a_\alpha+a_\beta}{ f}\nonumber\\
&=(1+\dfrac1f)(a_\alpha+a_\beta)-(R_{\alpha4\alpha4}+R_{\beta4\beta4})-(1+\dfrac{1}{2 f})\frac12-\dfrac{a_\alpha^2+a_\beta^2}{ f}-\dfrac{a_\alpha a_\beta}{f}.\label{r}
\end{align}
Then we have 
\begin{align}\label{eq-prod}
\qquad \qquad\sum_{\alpha\ne \beta}^3R_{\alpha\beta\alpha\beta}a_\alpha a_\beta
&=(1+\frac1f)\sum_{\alpha\ne \beta}^3(a_\alpha+a_\beta)a_\alpha a_\beta-\frac12(1+\frac1{2f})\sum_{\alpha\ne \beta}^3a_\alpha a_\beta\nonumber\\
&\quad-2\sum_{\alpha\ne \beta}^3R_{\alpha4\alpha4}a_\alpha a_\beta-\frac2{ f}\sum_{\alpha\ne \beta}^3a_\alpha^3a_\beta-\frac1{ f}\sum_{\alpha\ne \beta}^3a_\alpha^2a_\beta^2.
\end{align}
Now we compute the sums in  the right of  \eqref{eq-prod} respectively. We will use the equalities $\sum_{\alpha=1}^3a_{\alpha}=1$ and $\sum_{\alpha=1}^3a_{\alpha}^2=\frac12$.

\begin{align}\label{term0}
\sum_{\alpha\ne \beta}^3a_\alpha a_\beta
=\sum_{\alpha=1}^3\sum_{\beta\ne \alpha}a_\alpha a_\beta
=\sum_{\alpha=1}^3a_\alpha(1-a_\alpha)
=1-\frac12=\frac12.
\end{align}

\begin{align}\label{term1}
\sum_{\alpha\ne\beta}^3(a_\alpha+a_\beta)a_\alpha a_\beta&=2\sum_{\alpha=1}^3\sum_{\beta\ne \alpha}a_\alpha^2a_\beta\nonumber\\
&=2\sum_{\alpha=1}^3a_\alpha^2(1-a_\alpha)
=1-2\sum_{\alpha=1}^3a_\alpha^3.
\end{align}

\begin{align}\label{term3}
\sum_{\alpha\ne \beta}^3a_\alpha^3a_\beta
&=\sum_{\alpha=1}^3\sum_{\beta\ne \alpha}a_\alpha^3a_\beta
=\sum_{\alpha=1}^3a_\alpha^3(1-a_\alpha)=\sum_{\alpha=1}^3a_\alpha^3-\sum_{\alpha=1}^3a_\alpha^4.
\end{align}

\begin{align}\label{term4}
\sum_{\alpha\ne \beta}^3a_\alpha^2a_\beta^2
=\sum_{\alpha=1}^3\sum_{\beta\ne \alpha}a_\alpha^2a_\beta^2
=\sum_{\alpha=1}^3a_\alpha^2(\frac12-a_\alpha^2)
=\frac14-\sum_{\alpha=1}^3a_\alpha^4.
\end{align}

\noindent We calculate $\sum_{\alpha\ne \beta}^3R_{\alpha4\alpha4}a_\alpha a_\beta$ as follows. Since $S$ is constant,
(\ref{eq3}) becomes
\[
-\sum_{k,l=1}^4R_{ikjl}f_kf_l=\sum_{k=1}^4(\nabla_kf_{ij})f_k+\sum_{k=1}^4R_{kj}(R_{ik}-\frac12 g_{ik}).
\]
Note  $f_4=\nabla_{\frac{\nabla f}{|\nabla f|}}f=|\nabla f|$ and $ f_{\alpha}=0$.  The above equation yields that, for  $\alpha$ fixed,
\begin{align}\label{eq-ric-I4}
R_{\alpha4\alpha4}|\nabla f|^2&=-(\nabla_4f_{\alpha\alpha})|\nabla f|+\sum_{k=1}^4R_{\alpha k}(\frac12\delta_{k\alpha}-R_{k\alpha})\nonumber\\
&=(\nabla_4R_{\alpha\alpha})|\nabla f|+\frac12 a_\alpha-a_\alpha^2.
\end{align}
In the last equality of \eqref{eq-ric-I4}, we have used $R_{4\alpha}=0$ and $R_{\alpha\beta}=a_{\alpha}\delta_{\alpha\beta}$.
Then
\begin{align}\label{eq-ric-sum}
\sum_{\alpha\ne \beta}^3R_{\alpha4\alpha4}a_\alpha a_\beta
&=\sum_{\alpha=1}^3\sum_{\beta\ne \alpha}R_{\alpha4\alpha4}a_\alpha a_\beta\nonumber\\
&=\sum_{\alpha=1}^3R_{\alpha4\alpha4}a_\alpha(1-a_\alpha)\nonumber\\
&=\frac1{|\nabla f|^2}\sum_{\alpha=1}^3\big[(\nabla_4R_{\alpha\alpha})|\nabla f|+\frac12 a_\alpha-a_\alpha^2\big]a_\alpha(1-a_\alpha)\nonumber\\
&=\frac1{|\nabla f|}\sum_{\alpha=1}^3\big[(\nabla_4R_{\alpha\alpha})(a_\alpha-a_\alpha^2)+\frac1{|\nabla f|^2}\sum_{\alpha=1}^3(\frac12a_\alpha^2 -\frac32 a_\alpha^3+a_\alpha^4)
\end{align}
Besides, by $R_{\alpha\beta}=a_\alpha\delta_{\alpha\beta}$  and $a_4=0$,  it holds that,
\[\nabla_4\textrm{tr}(\textrm{Ric}^3)=3\sum_{\alpha=1}^3(\nabla_4 R_{\alpha\alpha})a_{\alpha}^2,
\]
and
\[
0=\nabla_4|\text{Ric}|^2=2\sum_{\alpha=1}^3(\nabla_4R_{\alpha\alpha})a_{\alpha}.
\]
Then substituting these two equalities above into \eqref{eq-ric-sum} yields
\begin{align}\label{eq-ric-sum2}
\sum_{\alpha\ne \beta}^3R_{\alpha4\alpha4}a_\alpha a_\beta&=-\frac1{3|\nabla f|}(\nabla_4\textrm{tr}(\textrm{Ric}^3))+\frac1{|\nabla f|^2}\sum_{\alpha=1}^3(\frac12a_\alpha^2 -\frac32 a_\alpha^3+a_\alpha^4)\nonumber\\
&=-\frac1{3f}(\nabla f) (\textrm{tr}(\textrm{Ric}^3))+\frac{1}{4f}-\frac{3}{2f}\sum_{\alpha=1}^3a_\alpha^3+\frac1{ f}\sum_{\alpha=1}^3a_\alpha^4.
\end{align}
In the second equality of \eqref{eq-ric-sum2}, we have used $\displaystyle e_4=\dfrac{\nabla f}{|\nabla f|}$ and $|\nabla f|^2=f$. 
Substituting \eqref{term0}, \eqref{term1}, \eqref{term3}, \eqref{term4} and \eqref{eq-ric-sum2} into \eqref{eq-prod} gives
\begin{align}\label{eq-product2}
\sum_{\alpha\ne \beta}^3R_{\alpha\beta\alpha\beta}a_\alpha a_\beta&=(1+\frac1f)(1-2\sum_{\alpha=1}^3a_\alpha^3)-\frac14(1+\frac1{2f})\nonumber\\
&\qquad-\frac2{ f}(\sum_{\alpha=1}^3a_\alpha^3-\sum_{\alpha=1}^3a_\alpha^4)-\frac1{ f}(\frac14-\sum_{\alpha=1}^3a_\alpha^4)\nonumber\\
&\qquad\quad
-2\bigg[-\frac1{3f}(\nabla f) (\textrm{tr}(\textrm{Ric}^3))+\frac{1}{4f}-\frac{3}{2f}\sum_{\alpha=1}^3a_\alpha^3+\frac1{ f}\sum_{\alpha=1}^3a_\alpha^4\bigg]\nonumber\\
 &=\frac34+\frac2{3f}(\nabla f) (\textrm{tr}(\textrm{Ric}^3))+\frac1{8f}-(2+\frac1f)\sum_{\alpha=1}^3a_\alpha^3+\frac1{ f}\sum_{\alpha=1}^3a_\alpha^4.
\end{align}
Noting that  $R_{4i}=0, i=1, \ldots,  4,$ and  $R_{\alpha\beta}=a_{\alpha}\delta_{\alpha\beta}$ at $p$, we have 
\[
\sum_{i,j,k,l=1}^4R_{ij}R_{ikjl}R_{kl}=\sum_{\alpha\ne \beta}^3R_{\alpha\beta\alpha\beta}a_\alpha a_\beta.
\]
By \eqref{ric-square-1},  it holds that
\begin{align}\label{nabla-ric}
\frac12+|\nabla \text{Ric}|^2-2\sum_{\alpha\ne \beta}^3R_{\alpha\beta\alpha\beta}a_\alpha a_\beta=0,
\end{align}
Using   \eqref{eq-product2} and \eqref{nabla-ric},    the inequality in Proposition \ref{tr}-(ii) implies  that
\begin{align}\label{deltaRic}
\Delta_f\textrm{tr}(&\text{Ric}^3)\nonumber\\
&\geq 4\textrm{tr}(\textrm{Ric}^3)-1\nonumber\\
&\quad-\frac32\big[1+\frac4{3f}(\nabla f) (\textrm{tr}(\textrm{Ric}^3))+\frac1{4f}-2(2+\frac1f)\sum_{\alpha=1}^3a_\alpha^3+\frac2{ f}\sum_{\alpha=1}^3a_\alpha^4\big]\nonumber\\
&= -\frac52-\frac{3}{8f} -\frac2{f}(\nabla f) (\textrm{tr}(\textrm{Ric}^3)) +(10+\frac3f)\sum_{\alpha=1}^3a_\alpha^3-\frac{3}f\sum_{\alpha=1}^3a_\alpha^4.
\end{align}
We will use an algebraic equality in the following (see its proof in the appendix of this paper): If $\sum_{\alpha=1}^3a_\alpha=1$ and $\sum_{\alpha=1}^3a_\alpha^2=\frac12$, then
\begin{align}\label{cor-s4-1}
\sum_{\alpha=1}^3a_\alpha^4=-\frac{5}{24}+\frac43 \sum_{\alpha=1}^3a_\alpha^3.
\end{align}
By \eqref{cor-s4-1}, \eqref{deltaRic} becomes
\begin{align}\label{eq-laplace}
\qquad\qquad&\Delta_f\textrm{tr}(\text{Ric}^3)\nonumber\\
&\geq -\frac52+\frac{1}{4f} -\frac2{f}(\nabla f) (\textrm{tr}(\textrm{Ric}^3)) +(10-\frac1f)\sum_{\alpha=1}^3a_\alpha^3\nonumber\\
&=(10-\frac1{f})\big(\textrm{tr}(\textrm{Ric}^3)-\frac14\big)-\frac2{f}\langle \nabla f,\nabla \textrm{tr}(\textrm{Ric}^3)\rangle.
\end{align}
Therefore,
\begin{align}\label{ineq-lem}
\qquad\quad&\Delta_f\big[f\big(\textrm{tr}(\textrm{Ric}^3)-\frac14\big)\big]\nonumber\\
&=f\Delta_f(\textrm{tr}(\textrm{Ric}^3))+2\langle \nabla f,\nabla \textrm{tr}(\textrm{Ric}^3)\rangle+\big(\textrm{tr}(\textrm{Ric}^3)-\frac14\big)\Delta_ff\nonumber \\
&\geq (10 f-1)\big(\textrm{tr}(\textrm{Ric}^3)-\frac14\big)-2\langle \nabla f,\nabla \textrm{tr}(\textrm{Ric}^3)\rangle\nonumber\\
&\qquad+2\langle \nabla f,\nabla \textrm{tr}(\textrm{Ric}^3)\rangle+\big(\textrm{tr}(\textrm{Ric}^3)-\frac14\big)(1-f)\nonumber\\
&=9 f\big(\textrm{tr}(\textrm{Ric}^3)-\frac14\big).
\end{align}
In  the last equality above, we have used $\Delta_ff=\dfrac{n}{2}- f-S=1- f$.   We have proved that \eqref{ineq-lem} holds for $p\in M$ with $f(p)\neq 0$. Further, since the focal variety $M_{-}=f^{-1}(0)$ is a smooth submanifold with codimension $2$,  \eqref{ineq-lem} holds on $M$.  We complete the proof.

\end{proof}

Now let us prove our main result, that is 

\begin{theorem}(Theorem \ref{thm-rig}) \label{thm-rig-1} 
 Let $(M, g, f)$
be a $4$-dimensional  complete noncompact gradient shrinking Ricci soliton satisfying \eqref{soliton}. If $M$  has the constant scalar curvature $S=2\lambda$, then 
it must be isometric to a finite quotient of $\mathbb{S}^2\times \mathbb{R}^2$. 
\end{theorem}

\begin{proof} Since $M$ has finite fundamental group, we only need to prove the theorem assuming that $M$ is simply connected. For convenience, choose $\lambda=\frac12$ and  let $M$ satisfies \eqref{soliton-3} and \eqref{f-norm} with the scalar curvature $S=1$. We   use the same  notation as in the proof of Proposition \ref{tr}.\\

We have known that  $a_1+a_2+a_3=1$ and   $a_1^2+a_2^2+a_3^2=\dfrac12$ imply $a_i\geq 0$, $i=1, 2, 3$. Hence \eqref{product} implies that 
\begin{align}\label{trRic}\textrm{tr}(\text{Ric}^3)-\frac14=3a_1a_2a_3\geq 0,
\end{align}
and the equality holds if and only if $a_1a_2a_3=0$, namely, at least one of $a_1, a_2$ and $ a_3$ is zero. Assume $a_3=0$.  Then from $S=1$ and $|\textrm{Ric}|^2=\frac12$, we have $a_1=a_2=\frac12$. \\

Since $\textrm{tr}(\text{Ric}^3)$ is bounded and $f$ is of quadratic growth, 

\[\int _{D(2r)\backslash D(r)} \big(\textrm{tr}(\text{Ric}^3)-\frac14\big)^2f^2e^{-f}\leq Cr^4e^{-\frac{r^2}4}\textrm{vol}(D(2r))\leq Cr^{6}e^{-\frac{r^2}4},\]


\noindent where $D_r=\{x\in M; \dist (x, M_{-})<r\}$. Here we have used (v) of Theorem \ref{thm-cscs}. Alternatively, we may  use the volume estimate by Cao and the second author of the present paper in \cite{MR2732975}, which says that 
$\textrm{vol}(D(r))\leq cr^n$ for any $n$-dimensional complete gradient shrinking Ricci soliton. Hence,

\[\int _{M} (\textrm{tr}\big(\text{Ric}^3)-\frac14\big)^2f^2e^{-f}<\infty.
\]

\noindent We  also need confirm the finiteness of  $ \int_M\big|\nabla
[f(\textrm{tr}(Ric^3)-\frac14)]\big|^2e^{-f}.$ In fact, it was proved by Munteanu and Wang (\cite{MW15}, Theorem 1.2)  that for a four-dimensional complete shrinking gradient Ricci soliton $M$ with bounded
scalar curvature $S$,  there exists a constant $c>0$ such that
\begin{align}|Rm|\leq cS, \quad \text{on} \quad M.
\end{align}
This implies that the curvature operator $Rm$ is bounded if $S$ is constant.  By \eqref{ric-square-1}, we get that $|\nabla \text{Ric}|$ is bounded. Therefore, by a direct computation,  the boundedness of $|\text{Ric}|$ and $|\nabla \text{Ric}|$ implies  that $|\nabla \textrm{tr}(\text{Ric}^3)|$ is also bounded. Then
\begin{align*}
\int_M&\big|\nabla
\big[f\big(\textrm{tr}(\text{Ric}^3)-\frac14\big)\big]\big|^2e^{-f}\\
&\leq 2\int_M\big(\textrm{tr}(\text{Ric}^3)-\frac14\big)^2\big|\nabla f\big|^2e^{-f}+2\int_M \big|\nabla
\big(\textrm{tr}(\text{Ric}^3)-\frac14\big)\big|^2f^2e^{-f}\\
&<\infty.
\end{align*}

\bigskip
\noindent Denote $h=f\big(\textrm{tr}(Ric^3)-\frac14\big)$. From Lemma \ref{lemma2}, we have

\begin{align}\label{ineq-soma}
-\Delta_fh+9 h\leq 0.
\end{align} 

\noindent  Let $\phi$  be the nonnegative $C^{\infty}$ function on $M$ with compact support satisfying  that  $\phi$ is $1$ on $B_R$,  $|\nabla\phi|\leq C$ on $B_{R+1}\setminus B_R$, and $\phi=0$ on $M\setminus B_{R+1}$, where $B_R$ denotes the geodesic ball of $M$ of radius $R$ centered at  a fixed $p\in M$.  Multiplying both sides of \eqref{ineq-soma} by $\phi^2 he^{-f}$ and integrating on $M$, we have

\[\begin{split}
0&\geq-\int_M\phi^2 h
(\Delta_fh)e^{-f}+9\int_M\phi^2h^2e^{-f}\\
&=\int_M \phi^2|\nabla h|^2 e^{-f} +\int_M2\langle h\nabla\phi,\phi\nabla h\rangle e^{-f}+9\int_M\phi^2h^2e^{-f}\\
&\geq \int_M \phi^2|\nabla h|^2 e^{-f} -\frac12 \int_M\phi^2|\nabla h|^2 e^{-f}-2\int_Mh^2|\nabla \phi|^2 e^{-f} +9\int_M\phi^2h^2e^{-f}\\
&\geq \frac12 \int_{B_R} |\nabla h|^2 e^{-f}-2C\int_{B_{R+1}\backslash B_R}h^2 e^{-f} +9\int_{B_R}h^2e^{-f}
\end{split}\]
Noting the finiteness of $\int_Mh^2 e^{-f}$ and letting $R\rightarrow \infty$, we get 
\begin{align}\label{vanish} 
0\geq \frac12 \int_{M} |\nabla h|^2 e^{-f}+9\int_Mh^2e^{-f}\geq 0.
\end{align}
From \eqref{vanish},  $h=f\big(\textrm{tr}(Ric^3)-\frac14\big)\equiv 0$ on $M$. Noting that the focal variety $M_{-}=f^{-1}(0)$ is a $2$-dimensional smooth submanifold, we  can get that $\textrm{tr}(\text{Ric}^3)-\frac14\equiv 0$ on $M$. By \eqref{trRic}, one of $a_{i}, i=1, 2, 3,$ must be zero. Then this together with the equalities $\sum_{i=1}^3a_{i}=1$ and $\sum_{i=1}^3a_{i}^2=\frac12$ implies that   two of $a_1, a_2$ and $ a_3$ must be $\frac12$. Therefore the eigenvalues of Ricci curvature can be arranged so that $a_1= a_2=\frac12, a_3=a_4=0.$ This implies that the rank of Ricci tensor is $2$.  Hence $M$ must be rigid (\cite[Theorem 2]{FR}). Since $S=1$ and the dimension of the soliton $M$ is four,  it must be the quotient of $\mathbb{S}^2\times\mathbb{R}^2$.
The proof is completed.

\end{proof}
\begin{remark} We would like to point out that the  finiteness of $ \int_{M} |\nabla h|^2 e^{-f}$ can also be obtained directly from computations for (\ref{vanish}). Our argument  here will be useful for expanding gradient Ricci solitons.
\end{remark}
 
\section{Appendix}
For the sake of completeness, we include the proof of  the following algebraic identities used in this  article. 
\begin{proposition}\label{lem4.1}
Let 
\[\sigma_i=a^i+b^i+c^i,
\]
where $i=1,2,3,4$.
Then 
\[
6abc=\sigma_1^3-3\sigma_1\sigma_2+2\sigma_3,
\]
and
\[\sigma_4=\frac{\sigma_1^4}6-\sigma_1^2\sigma_2+\frac{\sigma_2^2}2+\frac{4\sigma_1\sigma_3}3.
\]
\end{proposition}
\begin{proof}The first equation  is
\[
6abc=(a+b+c)^3-3(a+b+c)(a^2+b^2+c^2)+2(a^3+b^3+c^3),
\]
which follows from a direct computation.\\

Now we prove the second one. Since 
\[
2(ab+bc+ac)=\sigma_1^2-\sigma_2
\]
\[
\text{and} \quad (ab+bc+ac)^2=a^2b^2+a^2c^2+b^2c^2+2abc(a+b+c),
\]

we have 
\[\begin{split}
\sigma_4&=(a^2+b^2+c^2)^2-2(a^2b^2+a^2c^2+b^2c^2)\\
&=\sigma_2^2-2\big(\frac{(\sigma_1^2-\sigma_2)^2}4-2abc\sigma_1\big)\\
&=\sigma_2^2-\frac{(\sigma_1^2-\sigma_2)^2}2+4abc\sigma_1\\
&=\sigma_2^2-\frac{\sigma_1^4-2\sigma_1^2\sigma_2+\sigma_2^2}2+\frac23(\sigma_1^3-3\sigma_1\sigma_2+2\sigma_3 )\sigma_1\\
&=\frac{\sigma_1^4}6-\sigma_1^2\sigma_2+\frac{\sigma_2^2}2+\frac{4\sigma_1\sigma_3}3.
\end{split}
\]

\end{proof}
\begin{cor}\label{cor-s4} If $a_1+a_2+a_3=S$  and $a_1^2+a_2^2+a_3^2=\frac{S^2}2$, then
\[
\sum_{\alpha=1}^3a_\alpha^4=-\frac{5S^4}{24}+\frac{4S}3 \sum_{\alpha=1}^3a_\alpha^3.
\]
\end{cor}

\end{document}